\documentclass[12pt]{article}

\usepackage{amsmath}
\usepackage{amsthm,amscd,amsfonts}
\usepackage{amssymb, upref, color}
\usepackage{amsmath,amssymb,amsthm,mathrsfs}
\usepackage{color}
\usepackage[dvips]{graphicx}
\usepackage{epsf,epsfig, verbatim} 
\usepackage{latexsym,bm}

\numberwithin{equation}{section} 

\theoremstyle{plain}
\newtheorem{exam}{Example}[section]
\newtheorem{theorem}[exam]{Theorem}

\newtheorem{remark}[exam]{Remark}

\newtheorem{definition}[exam]{Definition}
\newtheorem{corollary}[exam]{Corollary}

\begin{document}
\date{}
\title{
Sharpness of  $C^0$ conjugacy for the non-autonomous differential equations with Lipschitzian perturbation
\footnote{ This paper was jointly supported from the National Natural
Science Foundation of China under Grant (No. 11671176, 11931016) and  Grant Fondecyt 1170466.}
}
\author
{
Weijie Lu$^{a}$\,\,\,\,\,
Manuel Pinto$^{b}$\,\, \,\,
Y-H. Xia$^{a}\footnote{Corresponding author. Yonghui Xia, xiadoc@outlook.com;yhxia@zjnu.cn. 
}$
\\
{\small \textit{$^a$ College of Mathematics and Computer Science,  Zhejiang Normal University, 321004, Jinhua, China}}\\
{\small \textit{$^b$ Departamento de Matem\'aticas, Universidad de Chile, Santiago, Chile }}\\
{\small Email: luwj@zjnu.edu.cn;  pintoj.uchile@gmail.com; xiadoc@outlook.com; yhxia@zjnu.cn.}
}

\maketitle

\begin{abstract}
  The classical $C^0$ linearization
  theorem for the non-autonomous differential equations
   states the existence of a $C^0$ topological conjugacy between the nonlinear system and its linear part. That is, there exists a homeomorphism (equivalent function) $H$ sending the solutions of the nonlinear system onto those of its linear part.
It is proved in the previous literature that the equivalent function $H$ and its inverse $G=H^{-1}$ are both H\"{o}lder continuous if the nonlinear perturbation is Lipschitzian. Questions: {\bf is it possible to improve the regularity? Is the regularity sharp?} To answer this question, we construct a counterexample to show that the equivalent function
$H$ is exactly Lipschitzian, but the inverse $G=H^{-1}$ is merely H\"{o}lder continuous. Furthermore, we propose a conjecture that such regularity of the homeomorphisms is sharp (it could not be improved anymore). We prove that {\color{blue} the conjecture is true for the systems with linear contraction}.
{Furthermore, we present the special cases of  linear perturbation, which are closely related to the spectrum.}
\end{abstract}

\section{Introduction}
\subsection{History of linearization}
    The well-known Hartman-Grobman theorem \cite{Hartman1,Grobman1} pointed out that nearby the hyperbolic equilibrium $u^{*}=0$,
the dynamic behaviour of the nonlinear system $u'=Au+f(u)$, where $f\in C^{1}(\mathbb{R}^{n})$, $f(x)=o(\|x\|)$,
is topologically conjugated to its linear problem $v'=Av$.
   Pugh \cite{Pugh1} initially studied global linearization and  presented a particular case of the Hartman-Grobman theorem as long as $f$
satisfies some goodness conditions, like continuity, boundedness or being Lipschitzian.
   Moreover, the Hartman-Grobman theorem has been extended to many classes of differential equations such as retarded functional equations
by Farkas \cite{Farkas},
scalar reaction-diffusion equations by Lu \cite{Lu1},  Cahn-Hilliard equation and phase field equations by Lu and Bates \cite{Lu2} and
semilinear hyperbolic evolution equations by Hein and Pr\"{u}ss \cite{Hein-Pruss1}. Except for the $C^0$ linearization of the differential equations,
     Sternberg (\cite{Sternberg1,Sternberg2}) initially investigated $C^r$ linearization for $C^k$ diffeomorphisms.
Sell (\cite{Sell1}) extended the theorem of Sternberg. 
More mathematicians paid particular attention on the $C^1$ linearization.
 Belitskii \cite{Belitskii1},
ElBialy \cite{ElBialy1}, Rodrigues and Sol\`a-Morales \cite{RS-JDE} studied that
$C^1$ linearization of hyperbolic diffeomorphisms on Banach space independently.
   Recently, Zhang and Zhang \cite{ZWN-JFA} proved that the    $C^1$ linearization is H\"{o}lder continuous. 
   Zhang et al. \cite{ ZWN-MA,ZWN-ETDS,ZWN-JDE}  showed the sharp regularity of linearization for $C^{1}$ or $C^{1,1}$ hyperbolic diffeomorphisms, and it is proved that the   linearization is H\"{o}lder continuous.

    Efforts were made to the linearization of the nonautonomous dynamical system initially by Palmer \cite{Palmer1}.
   To weakened the conditions of Palmer's theorem,
   Jiang \cite{Jiang1} gave a version of  linearization result by  a generalized exponential dichotomy.
  Recently, Barreira and Valls \cite{B-V,B-V1,B-V2,B-V3,B-V4,B-V5} proved several versions of Hartman-Grobman theorem provided that the linear homogenous system admits a nonuniform contraction or nonuniform dichotomy. Xia et al. \cite{Xia-BSM} improved Palmer's linearization by assuming that the nonlinear perturbation would be unbounded or non-Lipschitzian.
    Huerta et al. \cite{Huerta2,Huerta1} constructed conjugacies between linear system and an unbounded nonlinear perturbation for a nonuniform contraction.
    { Inspired by the work of Reinfelds and \v{S}teinberga \cite{Reinfelds-IJPAM}, Backes et al. \cite{BDK-JDE} extended the linearization theorem to a class of non-hyperbolic system.
    Newly,  Backes and Dragi\v{c}evi\'{c} \cite{BD-arXiv2}  further gave some results of smooth linearization on the basis of \cite{BDK-JDE}.
  Furthermore, 
  in \cite{BD-arXiv1}, they presented the multiscale linearization of nonautonomous systems.
It is interesting that if the nonlinear perturbations are well-behaved, they do not  assume any asymptotic behaviour of the linear system. }
    Other generalizations of Hartman-Grobman theorem were given for
  the instantaneous impulsive system \cite{Reinfelds1,Fenner-Pinto,Xia1},
dynamic systems on time scales \cite{Potzche1}, differential equations with piecewise constant argument \cite{Zou-Xia,Pinto-depcag}.


  \subsection{ Counterexample and Motivation }

    In this paper, we pay particular attention to the regularity of the equivalent function $H$ which transforms the solutions of nonlinear system onto its linear part. It is proved in the previous works that the equivalent function and its inverse are H\"{o}lder continuous for the Hartman-Grobman theorem in $C^0$ linearization when the nonlinear perturbations are Lipschitzian (see e.g.
Hein and Pr\"{u}ss \cite{ Hein-Pruss1}, Jiang \cite{Jiang1}, Barreira and Valls \cite{B-V,B-V1,B-V2,B-V3,B-V5},Xia et al \cite{Xia-BSM}, Backes et al \cite{BDK-JDE}, Xia et al \cite{Xia1},Shi and Zhang \cite{Shi-Zhang1}).
    However, the following counterexample is to show that the equivalent function
$H$ is {\bf exactly Lipschitzian}, but the inverse $G=H^{-1}$ is merely H\"{o}lder continuous.

 \subsubsection{Counterexample to H\"older continuity of the transformation }

    \newtheorem{Exam}[exam]{Example}
\begin{Exam}\label{sharp-ex}
{\em
For the sake of readable and easily verification, we consider the scalar equation
\begin{equation}\label{n}
  x'=-x+f(x).
\end{equation}
We take the nonlinear perturbation $f$ as follows
\[
f(x)=
\begin{cases}
\epsilon, &  x\geq 1,\\
\epsilon x,&  0\leq x< 1,\\
 \epsilon x^{3},&  -1< x< 0,\\
 -\epsilon, &  x\leq -1,
\end{cases}
\]
where $0<\epsilon< \frac{1}{3}$.
 We see that
for all $x, x_{1}, x_{2}$,
\[ |f(x)|\leq\epsilon, \quad |f(x_{1})-f(x_{2})|\leq 3\epsilon|x_{1}-x_{2}|.\]
Moreover, the linear equation $x'=-x$ has an uniform contraction with constant $k=1$ and exponent $\alpha=1$.
Taking $k=1, L_{f}=3\epsilon, \alpha=1$ ($L_{f}K/\alpha=3\epsilon<1$), it is easy to show that there exists a  homeomorphism
$H(x)$ 
 transforming the solution $x(t)$ of \eqref{n} onto that of its linear part
\begin{equation}\label{l}
  x'=-x.
\end{equation}
To show the regularity of $H$ and  $G=H^{-1}$ (the inverse). We dive it into three steps.

\noindent {\bf Step 1}. We obtain the {\bf explicit expressions} of $H$ and  $G$.
Notice that
\[
-x+f(x)
\begin{cases}
  <0, & \mathrm{if}\; x>0, \\
  =0, & \mathrm{if}\; x=0, \\
  >0, & \mathrm{if}\;  x<0.
\end{cases}
\]
Thus a solution is either always $0$, always $<0$ or always $>0$. Now we divide the discussion into three cases.

\noindent {\em (I) zero solution.} $H(0)=0$ since $0$ is a solution of \eqref{n} and $H(0)$ is the unique solution of \eqref{l}. But $0$ is such a solution.

\noindent {\em (II) positive solution $x(t)>0$. } Clearly, $x(t)$ is strictly decreasing, i.e., $x(t)\rightarrow 0$ as $t\rightarrow +\infty$;
$x(t)\rightarrow \infty$ as $t\rightarrow -\infty$.
Therefore, there must exists a unique time $t_{0}$ such that $x(t_{0})=1$. Without loss of generality, we take $t_{0}=0$.
 If $t<0$, then $x(t)>1$ and so $x'(t)=-x(t)+\epsilon$ with $x(0)=1$.
 Hence,
 \[ x(t)=(1-\epsilon)e^{-t}+\epsilon,\quad t\leq 0.\]
If $t>0$, then $0<x(t)<1$ and so $x'(t)=-x(t)+\epsilon x(t)$ with $x(0)=1$. Hence,
\[ x(t)=e^{(-1+\epsilon)t}, \quad t\geq 0.\]
We need to find the unique solution $y(t)$ of \eqref{l} such that $|y(t)-x(t)|$ is bounded.
Looking at $x(t)$ when $t\leq 0$, we see that $y(t)=(1-\epsilon)e^{-t}$.
Hence for all $t$,
\[H(x(t))=(1-\epsilon)e^{-t}.\]
Then
\[ H(1)=H(x(0))=1-\epsilon.\]
If $\xi>1$, then there exists a unique time $ t<0$ such that $x(t)=(1-\epsilon)e^{-t}+\epsilon=\xi$. Then
\[ H(\xi)=H(x(t))=(1-\epsilon)e^{-t}=x(t)-\epsilon=\xi-\epsilon.\]
If $0<\xi<1$, then there exists a unique time $t>0$ such that $x(t)=e^{(-1+\epsilon)t}=\xi$. Then
\[ H(\xi)=H(x(t))=(1-\epsilon)e^{-t}=(1-\epsilon)x(t)^{\frac{1}{1-\epsilon}}
    =(1-\epsilon)\xi^{\frac{1}{1-\epsilon}}.\]
Therefore,
\[
H(x)=
\begin{cases}
  (1-\epsilon)x^{\frac{1}{1-\epsilon}}, &  0<x< 1, \\
  x-\epsilon, &  x\geq 1.
\end{cases}
\]

\noindent {\em (III) negative solution $x(t)<0$. } Clearly, $x(t)$ is strictly increasing, i.e., $x(t)\rightarrow 0$ as $t\rightarrow \infty$;
$x(t)\rightarrow -\infty$ as $t\rightarrow -\infty$.
So there must exists a unique time $t_{0}$ such that $x(t_{0})=-1$. Without loss of generality, we take $t_{0}=0$.
 If $t<0$, then $x(t)<-1$ and so $x'(t)=-x(t)-\epsilon$ with $x(0)=-1$.
 Hence,
 \[ x(t)=(\epsilon-1)e^{-t}-\epsilon,\quad t\leq 0.\]
If $t>0$, then $-1<x(t)<0$ and so $x'(t)=-x(t)+\epsilon x^{3}(t)$ with $x(0)=-1$. Letting $z=x^{-2}$, then $z'=2z-2\epsilon$ with $z(0)=1$. Hence,
$z(t)=(1-\epsilon)e^{2t}+\epsilon$, that is,
\[ x(t)=-\left[(1-\epsilon)e^{2t}+\epsilon\right]^{-\frac{1}{2}}, \quad t\geq 0.\]
We need to find the unique solution $y(t)$ of \eqref{l} such that $|y(t)-x(t)|$ is bounded.
Looking at $x(t)$ when $t\leq 0$, we see that $y(t)=(\epsilon-1)e^{-t}$.
Hence for all $t$,
\[H(x(t))=(\epsilon-1)e^{-t}.\]
Then
\[ H(1)=H(x(0))=\epsilon-1.\]
If $\xi<-1$, then there exists a unique time $ t<0$ such that $x(t)=(\epsilon-1)e^{-t}-\epsilon=\xi$. Consequently,
\[ H(\xi)=H(x(t))=(\epsilon-1)e^{-t}=x(t)+\epsilon=\xi+\epsilon.\]
If $-1<\xi<0$, then there exists a unique time $t>0$ such that $x(t)=-\left[(1-\epsilon)e^{2t}+\epsilon\right]^{-\frac{1}{2}}=\xi$. Consequently,
\[ H(\xi)=H(x(t))=(\epsilon-1)e^{-t}=-(1-\epsilon)^{\frac{3}{2}}\left(\frac{1}{(-x(t))^{2}}-\epsilon\right)^{-\frac{1}{2}}
    =-(1-\epsilon)^{\frac{3}{2}}\left(\frac{1}{(-\xi)^{2}}-\epsilon\right)^{-\frac{1}{2}}.\]
Thus  we obtain that
\[
H(x)=
\begin{cases}
  -(1-\epsilon)^{\frac{3}{2}}\left(\frac{1}{(-x)^{2}}-\epsilon\right)^{-\frac{1}{2}}, &  -1<x< 0, \\
  x+\epsilon, &   x\leq -1.
\end{cases}
\]
Summarizing the above three cases, we see that
\[
H(x)=
\begin{cases}
  x-\epsilon, &  x\geq 1,\\
   (1-\epsilon)x^{\frac{1}{1-\epsilon}}, &  0<x< 1, \\
  0, & x=0,\\
 -(1-\epsilon)^{\frac{3}{2}}\left(\frac{1}{(-x)^{2}}-\epsilon\right)^{-\frac{1}{2}}, &  -1<x< 0, \\
   x+\epsilon, &   x\leq -1,
\end{cases}
\]
where $0<\epsilon<\frac{1}{3}$.

\noindent {\bf Step 2}.  We  show that {\color{blue} $H$ is Lipschitzian, but it is not $C^{1}$}. It suffices to show that $H$ is continuous at $0$ and $\pm 1$, and $H'(x)$ is bounded,
but is not $C^{1}$ at $0$. In fact,\\
\indent {\bf(1)} $H(x)$ is continuous at $x=1$:
   \[\begin{split}
    &\lim\limits_{x\rightarrow 1^{-}} (1-\epsilon)x^{\frac{1}{1-\epsilon}}=1-\epsilon,\\
    &\lim\limits_{x\rightarrow 1^{+}} x-\epsilon=1-\epsilon.
   \end{split}\]
\indent {\bf(2)} $H(x)$ is continuous at $x=0$:
   \[\begin{split}
    &\lim\limits_{x\rightarrow 0^{+}} (1-\epsilon)x^{\frac{1}{1-\epsilon}}=0,\\
    &\lim\limits_{x\rightarrow 0^{-}}  -(1-\epsilon)^{\frac{3}{2}}\left(\frac{1}{(-x)^{2}}-\epsilon\right)^{-\frac{1}{2}}=0.
   \end{split}\]
\indent {\bf(3)} $H(x)$ is continuous at $x=-1$:
 \[\begin{split}
    &\lim\limits_{x\rightarrow -1^{+}} -(1-\epsilon)^{\frac{3}{2}}\left(\frac{1}{(-x)^{2}}-\epsilon\right)^{-\frac{1}{2}}=-1+\epsilon,\\
    &\lim\limits_{x\rightarrow -1^{-}}  x+\epsilon=-1+\epsilon.
   \end{split}\]
Hence, $H(x)$ is continuous, but the following fact proves that $H$ is not $C^{1}$ at $0$. Clearly,
for $0<x< 1$, $H'(x)=x^{\frac{\epsilon}{1-\epsilon}}$,
and for $ -1<x< 0$, $H'(x)=(1-\epsilon)^{\frac{3}{2}}(1-\epsilon (-x)^{2})^{-\frac{3}{2}}$. Therefore,
   \[\begin{split}
    &\lim\limits_{x\rightarrow 0^{+}} x^{\frac{\epsilon}{1-\epsilon}}=0,\\
    &\lim\limits_{x\rightarrow 0^{-}}  (1-\epsilon)^{\frac{3}{2}}(1-\epsilon (-x)^{2})^{-\frac{3}{2}}=(1-\epsilon)^{\frac{3}{2}},
   \end{split}\]
it implies that $H(x)$ is not in $C^{1}$.  \\
Fortunately, $H(x)$ is Lipschitz continuous,
since $H'(x)$ is continuous at $x$ except for $x=0$, and it is bounded with $|H'|\leq 1$.
Therefore, function $H(x)$ is globally Lipschitz continuous with Lipschitz constant $L = 1$, but is not in $C^{1}$.
Fig. 1 (H) illustrates the Lipschitzian continuity of $H$.

\noindent {\bf Step 3}. {\color{blue} The inverse homeomorphism $G=H^{-1}$ is H\"older continuous, but not Lipschitzian.}
In fact, the expression of the inverse homeomorphism $G=H^{-1}$ is given as
\[
G(y)=
\begin{cases}
y+\epsilon, &   y\geq 1-\epsilon,\\
(\frac{y}{1-\epsilon})^{1-\epsilon}, &  0<y< 1-\epsilon, \\
  0, & y=0,\\
 -\left((1-\epsilon)^{3}(-y)^{-2}+\epsilon\right)^{-\frac{1}{2}}, &  -1+\epsilon<y< 0, \\
y-\epsilon, &   y\leq \epsilon-1.
\end{cases}
\]
Obviously, $G(y)$ is continuous at $0$ and $\pm 1$. Thus, $G(y)$ is a continuous function. However, this is not Lipschitz continuous since $y^{1-\epsilon}$ is not Lipschitz, as $0<1-\epsilon<1$. $G'(y)\rightarrow \infty$ as $y\rightarrow0$. As you see in Fig. 1 (G), the curve of $G$ is tangent to vertical-axis $G(y)$ at $y=0$.
\begin{center}
\begin{tabular}{cc}
\epsfxsize=8cm \epsfysize=6cm \epsffile{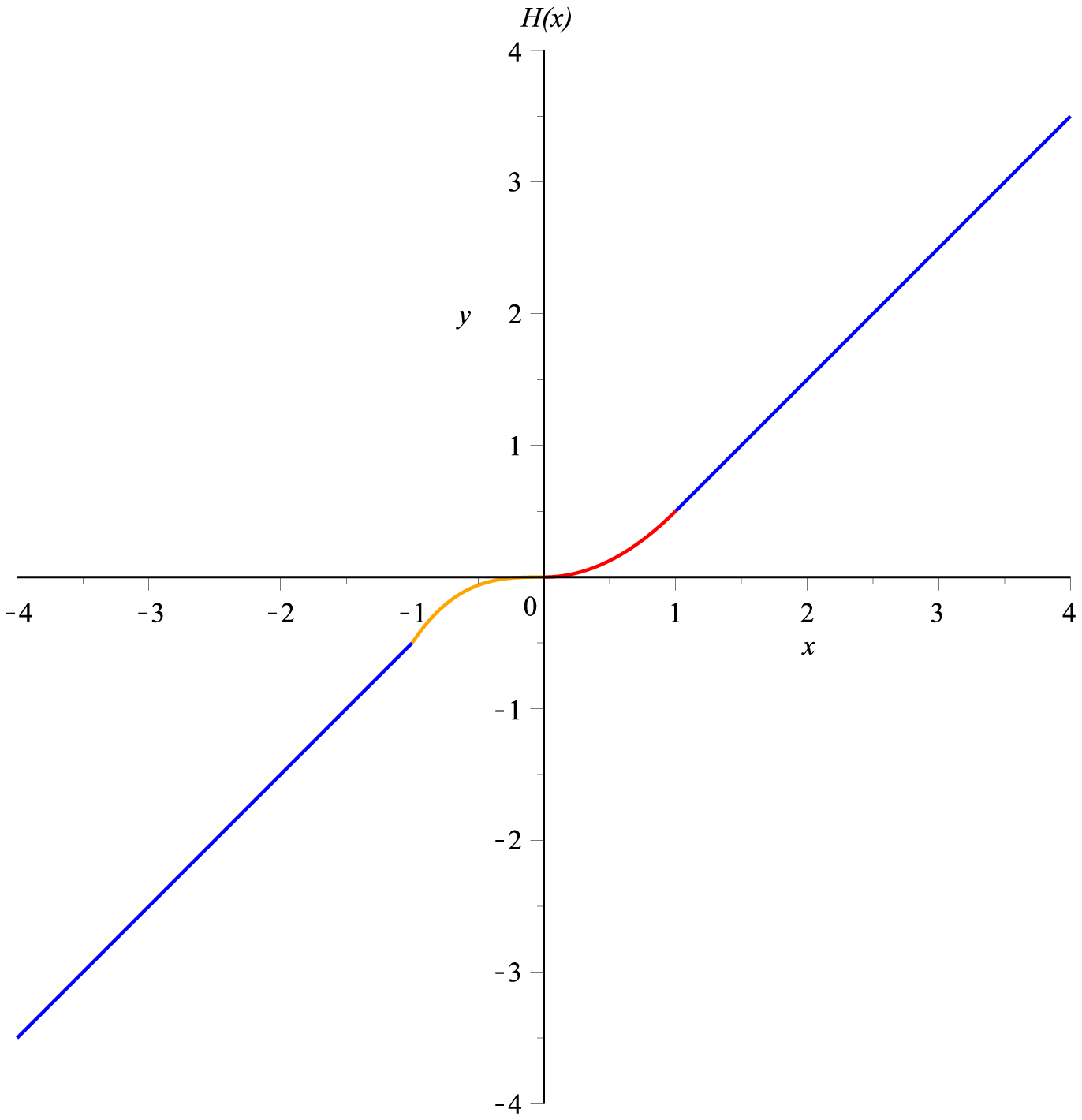}&
\epsfxsize=8cm \epsfysize=6cm \epsffile{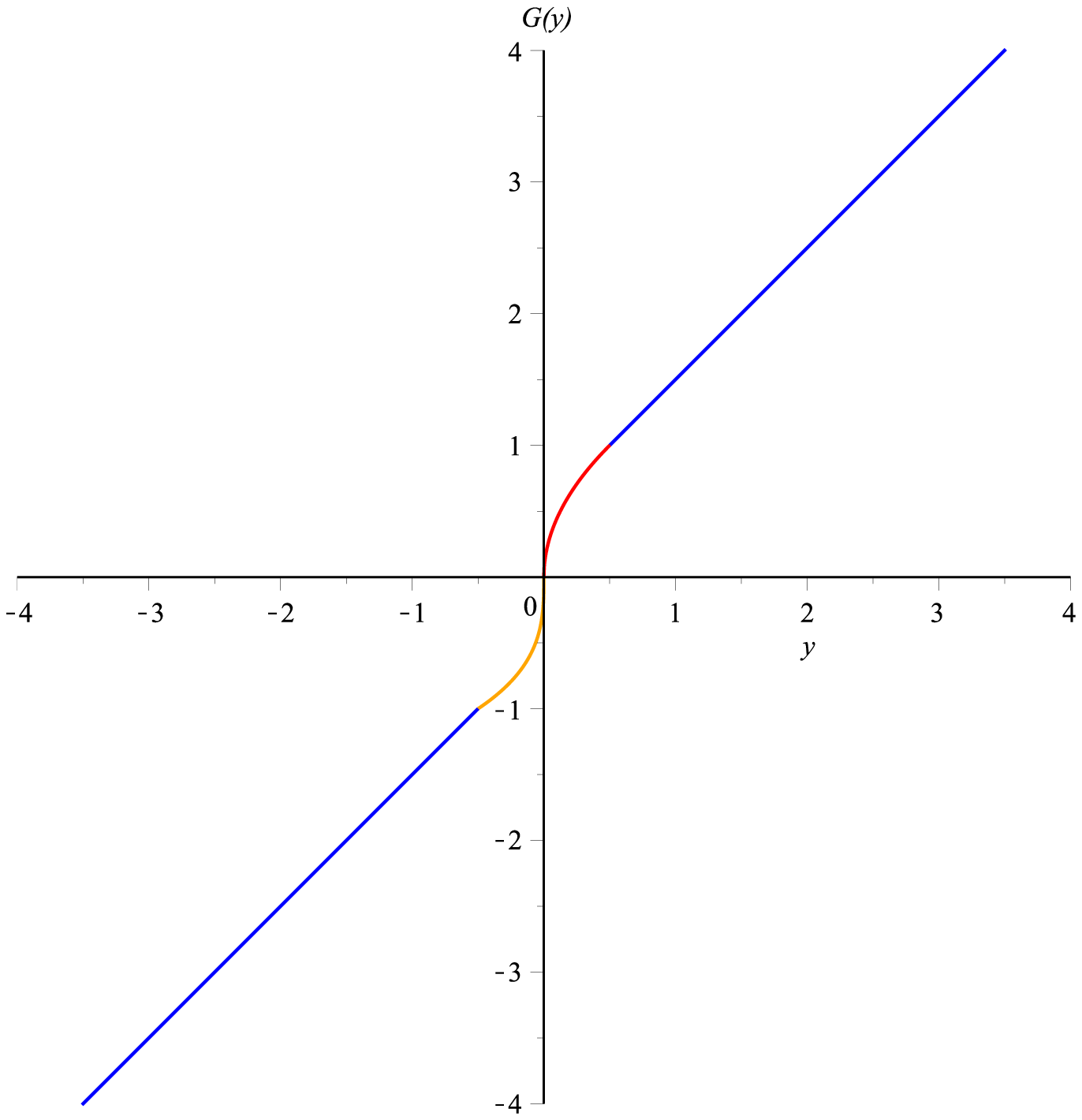}\\
\footnotesize{(H)} & \footnotesize{ (G) }
\end{tabular}
\end{center}
\begin{center}
\footnotesize {{Fig. 1} \ \ The equivalent function $H$ and its inverse $G$. See Fig. 1 (G), the curve of $G$ is tangent to vertical-axis $G(y)$ at $y=0$.}
\end{center}

}
\end{Exam}

 \subsubsection{Questions and motivation }

Counterexample above demonstrated that the topological equivalent function
$H$ could be {\bf exactly Lipschitzian} when the nonlinear perturbations are Lipschitzian.  And we find that the equivalent function $H$ and its inverse $G$ do not have the same smoothness. This new findings are different from the existing literature that
both the equivalent function and its inverse are H\"{o}lder continuous for the Hartman-Grobman theorem in $C^0$ linearization when the nonlinear perturbations are Lipschitzian (see e.g.
\cite{Hein-Pruss1,Jiang1,B-V,B-V1,B-V2,B-V3,B-V5,Xia-BSM,BDK-JDE,Xia1,Shi-Zhang1}).
  Motivated by Example \ref{sharp-ex}, two natural questions arise: is it possible to improve the regularity of the equivalent function to Lipschitz continuity for any systems with Lipschitzian nonlinear perturbations? And is  the regularity sharp?
     General speaking, these are not true. Then, it is very important to know how many and what of systems with Lipschitzian nonlinear perturbations have the property that the equivalent function is Lipschitzian.

    In the present paper, we give a positive answer if the linear system admits an exponential
 contraction. In this case, we show that $H$ is exactly Lipschitzian, but the inverse $G=H^{-1}$ is still H\"{o}lder continuous if the nonlinear perturbation is Lipschitzian.
   Furthermore, it is a sharp and it is impossible to be improved anymore.
     To the best of our knowledge, it is the first time to observe these facts in $C^0$ linearization with  Lipschitzian nonlinear perturbations.


   Secondly, we also discuss two special cases of the linear perturbations, which are closely related to the spectrum. If the perturbation is linear homogeneous, $H$ is $C^1$ and the inverse $G$ is H\"{o}lder continuous; If the perturbation is constant, both $H$ and $G$ are $C^{\infty}$.


\subsection{Outline of the paper}

 The rest of this paper is organized as follows: In Section 2, our main results are stated.
 In Section 3, some preliminary results are presented. In Section 4, rigorous proof is given to show the regularity of conjugacy.

\section{Statement of main results and illustrative examples}
Consider the following two nonautonomous systems:
   \begin{equation}\label{linear-eq}
     y'=A(t)y
   \end{equation}
and
   \begin{equation}\label{nonlinear-eq}
    x'=A(t)x+f(t,x),
   \end{equation}
where $A(t)$ is a bounded and continuous $n\times n$ matrix and $f:\mathbb{R}\times \mathbb{R}^{n}\rightarrow \mathbb{R}^{n}$ is a continuous function.




 \newtheorem{Remark}[exam]{Remark}

We introduce the operator $\mathcal{K}$ given by
\[  \mathcal{K}(\varphi)(t)=\int_{-\infty}^{t} U(t,s)\varphi(s)ds, \quad t\in\mathbb{R},\]
where $\varphi:\mathbb{R}\rightarrow\mathbb{R}^{n}$ is a function, $\|\mathcal{K}(\varphi)\|\leq\mathcal{L}(\|\varphi\|)$ with
\[ \mathcal{L}(b)(t)=\int_{-\infty}^{t} \exp\{-\alpha(t-s)\}b(s)ds,\]
for $b:\mathbb{R}\rightarrow (0,\infty)$ a continuous function.

Throughout this paper, we always suppose that \\
 ${\bf(A_{1})}$ there exist positive constants $k>0$ and $\alpha>0$ such that the transition matrix $U(t,s)=U(t)U^{-1}(s)$ of Eq. \eqref{linear-eq}
satisfies:
 \begin{equation}\label{contraction}
   \|U(t,s)\|\leq k \exp\{-\alpha(t-s)\},\quad \mathrm{for}\; t\geq s,
 \end{equation}
 that is, $U(t,s)$ is contractive.\\
${\bf(A_{2})}$ If there is a nonnegative local integrable functions $\mu(t)$ satisfying
$\sup_{t\in\mathbb{R}}\int_t^{t+1}\mu(s)ds \leq C_{1} $ ($C_{1}$ is a constant) such that for any $t\in\mathbb{R}$ and $x\in\mathbb{R}^{n}$, $\|f(t,x)\|\leq\mu(t)$;\\
 ${\bf(A_{3})}$ If there is a nonnegative local integrable functions $r(t)$ satisfying small enough
$C_{2}=\sup_{t\in\mathbb{R}}\int_t^{t+1}r(s)ds$ such that for any $t\in\mathbb{R}$ and $x_{1}, x_{2} \in\mathbb{R}^{n}$, $\|f(t,x_{1})-f(t,x_{2})\|\leq r(t)\|x_{1}-x_{2}\|$;\\
 ${\bf(A_{4})}$ $\sup_{t\in\mathbb{R}}\mathcal{L}(r)(t)=\theta<k^{-1}$. 

 We remark that if $\mu(t)\equiv \mu$ and $r(s)\equiv L_f$, ${\bf(A_{2})}-{\bf(A_{4})}$ reduce to the following particular cases:\\
 ${\bf(\widetilde{A}_{2})}$ the perturbation $f$ is continuous and bounded, i.e., for any $t\in\mathbb{R}$ and $x\in\mathbb{R}^{n}$, $\|f(t,x)\|\leq\mu<+\infty$;\\
 ${\bf(\widetilde{A}_{3})}$ for any $t\in\mathbb{R}$ and $x_{1}, x_{2} \in\mathbb{R}^{n}$, $\|f(t,x_{1})-f(t,x_{2})\|\leq L_{f}\|x_{1}-x_{2}\|$;\\
 ${\bf(\widetilde{A}_{4})}$ the Lipschitz constant $L_{f}$ satisfies: $L_{f}\leq \alpha/k$.





Next  theorems are the regularity of the topological equivalent function $H$ and its inverse $G=H^{-1}$.
According to the different classification of perturbation $f$, different regularity for three cases are presented as follows.

\noindent {\bf Case 1: the perturbation $f(t,x)$ is nonlinear (not linear) with respect to $x$.} 
\newtheorem{Theorem}[exam]{Theorem}
\begin{Theorem}\label{Thm-2}
 Suppose that the perturbation $f(t,x)$ is not linear with respect to $x$. If ${\bf(A_{1})}-{\bf(A_{4})}$  hold, then there is an equivalent function $H(t,x)$ (homeomorphism) transforming the solution of \eqref{nonlinear-eq}  onto that of linear part \eqref{linear-eq}.
 Furthermore, the equivalent function $H(t,x)$ and its inverse $G(t,y)$ fulfill:\\
(i) $H$ is exactly Lipschitzian,  that is, for any $x,\bar{x}\in\mathbb{R}^{n}$, there exists a positive constant $p_{1}>0$ such that
\[
\|H(t,x)-H(t,\bar{x})\| \leq p_{1} \|x-\bar{x}\|;
\]
(ii) the inverse $G=H^{-1}$ is H\"{o}lder continuous, that is, for any $y,\bar{y}\in\mathbb{R}^{n}$, there exist positive constants $p_{2}>0$ and $0<q<1$ such that
\[
 \|G(t,y)-G(t,\bar{y})\|\leq p_{2}\|y-\bar{y}\|^{q}.
\]
(iii) $H$ is Lipschitzian, but not $C^1$;  $G$ is H\"{o}lder continuous, but not Lipschitzian. The regularity of $H$ and $G$ can not be improved any more.
\end{Theorem}

\begin{corollary}
 If the perturbation $f(t,x)$ satisfies ${\bf({A}_{1})}$, ${\bf(\widetilde{A}_{2})}$, ${\bf(\widetilde{A}_{3})}$ and ${\bf(\widetilde{A}_{4})}$
hold, then all the conclusions of Theorem \ref{Thm-2} are valid.
\end{corollary}

Instead of the uniform exponential contraction, next theorem is for the nonuniform case, that is, the linear system admits a nonuniform exponential contraction (see Barreira and Valls   (\cite{B-V3,B-V4}).

{
\begin{theorem}\label{NED}
  Suppose that the linear system admits a nonuniform exponential contraction, Theorem \ref{Thm-2} is true for
$$\widetilde{\mu}(t)=\mu(t)\exp\{-\epsilon|t|\} \quad \mathrm{and} \quad \widetilde{r}(t)=r(t)\exp\{-\epsilon|t|\},$$
where $\mu(t)$ and $r(t)$ are given in $(A_2)$ and $(A_3)$, respectively.
\end{theorem}
\noindent
In fact, in view of $\widetilde{\mu}(t)=\mu(t)\exp\{-\epsilon|t|\}$, it is easy to see that
 \[ \mathcal{L}\{\exp\{\epsilon|\cdot|\}\widetilde{\mu}(\cdot)\}=\int_{-\infty}^{t}\exp\{-\alpha (t-s)\}\exp\{\epsilon|s|\}\widetilde{\mu}(s)ds
    =\mathcal{L}\{\mu(\cdot)\},
 \]
similarly, $\mathcal{L}\{\exp\{\epsilon|\cdot|\}\widetilde{r}(\cdot)\}=\mathcal{L}\{r(\cdot)\}$. Thus, all conditions of  Theorem \ref{Thm-2} in the nonuniform case are satisfied.
}

\begin{remark}
  Barreira and Valls \cite{B-V3,B-V4} constructed conjugacies between linear and non-linear non-uniform  contractions and proved that both the homeomorphisms are H\"older continuous. As Theorem \ref{NED} shows, the homeomorphism is Lipschitzian and its inverse is H\"older continuous for the non-uniform  contractions.   Thus, our result improves the known results in \cite{B-V3,B-V4}.
\end{remark}

For the autonomous systems, we have following corollary on the regularity of the conjugacy. That is, we consider the autonomous system
\begin{equation}\label{autoHG}
\dot{x}=Ax+f(x).
\end{equation}

\begin{corollary}
 Assume that the perturbation $f(x)$ satisfies $|f(x)|\leq\mu$, and $|f(x)-f(y)|\leq L_f\|x-y\|$,  and $L_f\leq \alpha/k$. If all the real parts of eigenvalues of $A$ are strictly less than 0,
 , then the equivalent function $H(x)$ and its inverse $G(y)$ fulfill:\\
(i) $H$ is exactly Lipschitzian,  that is, for any $x,\bar{x}\in\mathbb{R}^{n}$, there exists a positive constant $p_{1}>0$ such that
\[
\|H(x)-H(\bar{x})\| \leq p_{1} \|x-\bar{x}\|;
\]
(ii) the inverse $G=H^{-1}$ is H\"{o}lder continuous, that is, for any $y,\bar{y}\in\mathbb{R}^{n}$, there exist positive constants $p_{2}>0$ and $0<q<1$ such that
\[
 \|G(y)-G(\bar{y})\|\leq p_{2}\|y-\bar{y}\|^{q}.
\]
(iii) $H$ is Lipschitzian, but not $C^1$;  $G$ is H\"{o}lder continuous, but not Lipschitzian. The regularity of $H$ and $G$ can not be improved any more.
\end{corollary}

Now we pay particular attention to two special cases:

\noindent{\bf Case 2: $f$ is linear homogeneous perturbation.}

If $f(t,x)=B(t)x$, then equation \eqref{nonlinear-eq} reduces to
\begin{equation}\label{cor-eq}
  x'=A(t)x+B(t)x,
\end{equation}
where $B(t)$ is a continuous $n\times n$ matrix.

We need the following definition.
\begin{definition} [\cite{Sell-spectral}]
The exponential dichotomy spectrum of equation \eqref{linear-eq} is the set
\[\Sigma(A(t))=\{\gamma\in\mathbb{R}| x'=(A(t)-\gamma I)x\;\mathrm{ admits}\;\mathrm{ no}\;\mathrm{ exponential}\; \mathrm{dichotomy }\},\]
and the resolvent set $\rho(A(t))=\mathbb{R}\backslash \Sigma(A(t))$.
\end{definition}
It is known that the exponential dichotomy spectrum $\Sigma(A(t))$ of \eqref{linear-eq} is the disjoint union of $k$ closed intervals where
$0\leq k\leq n$. i.e., $\Sigma(A(t))=\bigcup_{i=1}^{k} [a_{i}, b_{i}]$, $a_{1}\leq b_{1}<a_{2}\leq b_{2}<\cdots <a_{k}\leq b_{k}$,
where $a_{1}$ could be $-\infty$ and $b_{k}$ could be $\infty$. If $k=0$, then $\Sigma(A(t))=\emptyset$ (see Sacker and Sell \cite{Sell-spectral}, Xia et al \cite{Xia-spectral,Xia-Zhang}).
Notice that instead of $f(t,x)$ being Lipschitz continuous with respect to $x$, 
it is differentiable with respect to $x$. Thus, a better regularity will be proved. Namely, $H$ is $C^{1}$ and its inverse $G$ is H\"{o}lder continuous.

\begin{theorem}\label{Cor3}
Suppose that $f(t,x)=B(t)x$ and \eqref{linear-eq} satisfies condition ${\bf(A_{1})}$ with $\Sigma(A(t))=\bigcup_{i=1}^{k} [a_{i}, b_{i}]$, where $b_{k}<0$.
  If there exist $\delta>0$ and  appropriately small $0\leq\epsilon< -b_{k}$ such that $\|B(t)\|\leq \delta$ and
   $\Sigma (A(t)+B(t))\subseteq \bigcup_{i=1}^{k} [a_{i}+\epsilon,b_{i}+\epsilon]$,
then the equivalent function $H(t,x)$ is $C^{1}$ and its inverse $G(t,y)$ is H\"{o}lder continuous.
\end{theorem}
{
\begin{Remark}

 In addition, if \eqref{linear-eq} satisfies condition ${\bf(\widetilde{A}_{1})}$ with $\Sigma(A(t))=\bigcup_{i=1}^{k} [a_{i}, b_{i}]$, where $a_1>0$,
and  there exist $\widetilde{\delta}>0$ and  appropriately small $0\leq\varepsilon< a_1$ such that $\|B(t)\|\leq \widetilde{\delta}$ and
   $\Sigma (A(t)+B(t))\subseteq \bigcup_{i=1}^{k} [a_{i}-\varepsilon,b_{i}-\varepsilon]$, then our result is similar to Theorem \ref{Cor3}.
\end{Remark}
}

\begin{Exam}\label{case2-ex}
{\em
\noindent
For Eq. \eqref{n}, we consider the nonlinear term $f=\epsilon x$, $0<\epsilon< 1$.
In this case, $\Sigma (A+B)=-1+\epsilon <0$ and $|B|=\epsilon $.
Similar to the procedure just shown, we have that
\[
H(x)=
\begin{cases}
   (1-\epsilon)x^{\frac{1}{1-\epsilon}}, &  x>0, \\
  0, & x=0,\\
(\epsilon-1)(-x)^{\frac{1}{1-\epsilon}}, &  x< 0.
\end{cases}
\]
This function $H(x)$ is at least in $C^{1}$, and $H(x)$ is Lipschitzian.
And the inverse function $G=H^{-1}$ is
\[
G(y)=
\begin{cases}
(\frac{y}{1-\epsilon})^{1-\epsilon}, &  y>0, \\
  0, & y=0,\\
 -(\frac{-y}{1-\epsilon})^{1-\epsilon}, &  y< 0.
\end{cases}
\]
This function $G(y)$ is H\"{o}lder continuous.
}
\end{Exam}

\noindent{\bf Case 3: $f$ is constant.}

If the nonlinear term $f(t,x)$ is a constant function, then equation \eqref{nonlinear-eq} reduces to a non-homogeneous linear equation.
Clearly,
\begin{remark}\label{Cor4}
If the nonlinear term $f(t,x)$ is a constant,
then there is a $C^{\infty}$ coordinate transformation such that Eq. \eqref{nonlinear-eq} is topologically conjugated to Eq. \eqref{linear-eq},
 namely, the equivalent function $H(t,x)$ and its inverse $G(t,y)$
are both $C^{\infty}$.
\end{remark}

\begin{Exam}\label{case3-ex}
{\em
 If $f(x)$ is a constant, then Eq. \eqref{n} reduces to
\[x'=-x+\delta. \]
In this case, linearization theorem is always satisfied. Moreover, it is easy to obtain that
\[H(x)=x-\delta, \quad G(y)=y+\delta.\]
Thus, $H, G$ are $C^{\infty}$.
}
\end{Exam}

\section{Preliminary results}
  In what follows, we always suppose that the conditions of Theorem \ref{Thm-2} are satisfied.
  Let $X(t,t_0,x_0)$ be a solution of Eq. \eqref{nonlinear-eq} satisfying the initial condition $X(t_0)=x_0$ and
    $Y(t,t_0,y_0)$ is a solution of Eq. \eqref{linear-eq} satisfying the initial condition $Y(t_0)=y_0$.

  Note that Xia et al \cite{Xia-BSM} has proved that if Eq. \eqref{linear-eq} admits an exponential dichotomy, and nonlinear perturbation $f$ satisfies ${\bf(A_{2})}-{\bf(A_{4})}$, then
  there exists a unique function $H: \mathbb{R}\times \mathbb{R}^{n}\rightarrow \mathbb{R}^{n}$ such that
  Eq. \eqref{nonlinear-eq} is topologically conjugated to Eq. \eqref{linear-eq}.
 Since the assumption ${\bf(A_{1})}$ (contraction) is a particular case of the exponential dichotomy, naturally, Eq. \eqref{nonlinear-eq} is topologically conjugated to Eq. \eqref{linear-eq}.

\newtheorem{Lemma}[exam]{Lemma}


Therefore, we can construct the two homeomorphisms as follows.
\begin{equation}
H(t,x)=x+h(t,(t,x)),\label{(H)}
\end{equation}
where
\begin{equation}\label{hhh}
   h(t,(\tau,\xi))=-\int_{-\infty}^t U(t,s)P(s)f(s,X(s,\tau,\xi))ds,
\end{equation}
is the unique bounded solution of the auxiliary equation
  \begin{equation*}\label{eq-z1}
    z'=A(t)z-f(t,X(t,\tau,\xi)),\,\,\,\mbox{for any fixed}\,\,\,\,(\tau,\xi).
  \end{equation*}
 And
\begin{equation}
G(t,y)=y+g(t,(t,y)),\label{(G)}
\end{equation}
where
\begin{equation}\label{ggg}
  g(t,(\tau,\xi))=\int_{-\infty}^t U(t,s)P(s)f(s,Y(s,\tau,\xi)+g(s,(\tau,\xi)))ds,
\end{equation}
 is the unique bounded solution of the auxiliary equation
\begin{equation*}\label{eq-z2}
  z'=A(t)z+f(t,Y(t,\tau,\xi)+z),\,\,\,\mbox{for any fixed}\,\,\,\,(\tau,\xi).
\end{equation*}

To prove the regularity of the homeomorphisms, we need the following lemma.
\begin{Lemma}\label{lemma4}
If $r(t)$ is a nonnegative local integrable function and $C_2=\sup\limits_{t\in \mathbb{R}}\int_t^{t+1}r(s)ds<\infty$, then we have
\begin{equation}\label{X-eq}
  \|X(t,t_0,x_{0})-X(t,t_0,\bar{x}_{0})\|\leq ke^{kC_{2}}\|x_{0}-\bar{x}_{0}\|\cdot e^{(kC_{2}-\alpha)|t-t_0|}.
\end{equation}
\end{Lemma}
\begin{proof}
We only prove the case of $t\geq t_0$, and the case of $t\leq t_0$ can be obtained analogously.
 Since $r(t)$ is a local integrable function, i.e., $C_{2}=\sup_{t\in\mathbb{R}}\int_{t}^{t+1}r(s)ds$, we have
\[\begin{split}
 \int_{t_0}^{t}r(s)ds =& \int_{t_0}^{1}r(s)ds+\int_{1}^{2}r(s)ds+\cdots +\int_{t-1}^{t}r(s)ds \\
                     \leq&  ([t-t_0]+1)C_{2},
\end{split}\]
where $[t-t_0]$ means taking the largest integer not great than $t-t_0$. Note that
\[ X(t,t_0,x_{0})=U(t,t_0)x_{0}+\int_{t_0}^{t}U(t,s)f(s,X(s,t_0,x_{0}))ds.\]
Then for any $x_{0}, \bar{x}_{0}\in\mathbb{R}^{n}$, the above equality leads to
\[\begin{split}
& \|X(t,t_0,x_{0})-X(t,t_0,\bar{x}_{0})\| \\
 \leq&  ke^{-\alpha (t-t_0)} \|x_{0}-\bar{x}_{0}\|
 +\int_{t_0}^{t} ke^{-\alpha(t-s)}r(s) \|X(s,t_0,x_{0})-X(s,t_0,\bar{x}_{0})\| ds.
 \end{split}\]
 Using the Bellman inequality, we have
\[\begin{split}
 \|X(t,t_0,x_{0})-X(t,t_0,\bar{x}_{0})\| \leq & ke^{\alpha t_0} \|x_{0}-\bar{x}_{0}\| e^{\int_{t_0}^{t} kr(s)ds-\alpha t} \\
             \leq &  ke^{kC_{2}} \|x_{0}-\bar{x}_{0}\|e^{(kC_{2}-\alpha)(t-t_0)}, \quad t\geq t_0.
\end{split}\]
The proof for  $t\leq t_0$  is similar.
\end{proof}

\section{Proofs of main results}
  Now we are in a position to prove our results.

\begin{proof}[Proof of Theorem \ref{Thm-2}.] We divide the proof of Theorem \ref{Thm-2} into four steps.\\
{\bf Step 1.}
We are going to prove the Lipschitz continuity of the equivalent function $H(t,x)$ by Lemma \ref{lemma4}. That is, we will show that
\[
\|H(t,x)-H(t,\overline{x})\|\leq p_{1}\|x-\overline{x}\|, \qquad p=\frac{1+k^{2}C_{2}e^{kC_{2}}-e^{-kC_{2}}}{1-e^{-kC_{2}}}.
\]
In fact, by uniqueness $X(t,(\tau,\xi))=X(t,(t,X(t,(\tau,\xi)))$. From \eqref{hhh}, it follows that
\[ h(t,(t,\xi))=-\int_{-\infty}^t U(t,s)f(s,X(s,t,\xi))ds.\]
Thus we get
\begin{eqnarray*}
 h(t,(t,\xi))-h(t,(t,\overline{\xi}))=-\int_{-\infty}^t U(t,s)(f(s,X(s,t,\xi))-f(s,X(s,t,\overline{\xi})))ds. \qquad \textbf{$(I_1)$}
\end{eqnarray*}
By condition {\bf (A$_2$)} and Lemma \ref{lemma4}, we have
\[\begin{split}
 \|I_{1}\| \leq & \int_{-\infty}^{t} ke^{-\alpha(t-s)} r(s)\cdot \|X(s,t,\xi)-X(s,t,\bar{\xi})\|ds \\
           \leq & \int_{-\infty}^{t} ke^{-\alpha(t-s)} r(s)\cdot ke^{kC_{2}}\|\xi-\bar{\xi}\|e^{(kC_{2}-\alpha)(s-t)}ds\\
            \leq & \frac{k^{2}C_{2}e^{kC_{2}}}{1-e^{-kC_{2}}}\cdot \|\xi-\bar{\xi}\|.
\end{split}\]
By the definition of $H(t,x)$,
\[
\begin{split}
 \|H(t,x)-H(t,\overline{x})\|\leq &\|x-\overline{x}\|+\frac{k^{2}C_{2}e^{kC_{2}}}{1-e^{-kC_{2}}}\|x-\overline{x}\|\\
 \leq& \frac{1+k^{2}C_{2}e^{kC_{2}}-e^{-kC_{2}}}{1-e^{-kC_{2}}} \|x-\overline{x}\|\\
 \equiv& p_{1}\|x-\overline{x}\|,
\end{split}
\]
which implies that $H$ is Lipschitzian.


\noindent {\bf Step 2.} We show that $H$ is not $C^1$. In fact, noticing $H=x+h$ and the expression of $h$ (see \eqref{hhh}), it is easy to see that
$h$ is not $C^1$ because $f$ is Lipschitzian, not $C^1$.

\noindent {\bf Step 3.} We claim that there exist positive constants $p_{2}>0$ and $0<q<1$ such that for all $t\in\mathbb{R}, y, \bar{y}\in\mathbb{R}^{n}$,
\[  \|G(t,y)-G(t,\bar{y})\|\leq p_{2}\|y-\bar{y}\|^{q}.\]
Usually this point is treated with successive approximations.
 From  \eqref{ggg}, we know that $g(t,(\tau,\xi))$ is a fixed point of the following map $\mathcal{T}$
\begin{equation}\label{g-expression}
\begin{split}
 (\mathcal{T}z)(t)=&\int_{-\infty}^t U(t,s)f(s,Y(s,\tau,\xi)+z(s))ds.
 \end{split}
\end{equation}
Let $g_{0}(t,(\tau,\xi))\equiv 0$, and by recursion define
\[\begin{split}
 g_{m+1}(t,(\tau,\xi))=& \int_{-\infty}^t U(t,s)f(s,Y(s,\tau,\xi)+g_{m}(s,(\tau,\xi)))ds.
\end{split} \]
  It is not difficult to show that
\[ g_{m}(t,(\tau,\xi))\rightarrow  g(t,(\tau,\xi)), \quad \text{as} \ \ m\rightarrow +\infty, \]
  uniformly with respect to $t, \tau,\eta$.
\\
  Note that $g_{0}(t,(\tau,\xi))=g_{0}(t,(t,Y(t,\tau,\xi)))$.
  By induction, we see that
 $g_{m}(t,(\tau,\xi))=g_{m}(t,(t,Y(t,\tau,\xi)))$, for all $m (m\in\mathbb{N})$.
Choose a sufficiently large positive constant $\lambda0$ and a sufficiently small positive constant $0<q<1$ such that
\begin{equation}\label{4.1}
  \begin{cases}
  \lambda>\frac{3kC_{1}}{1-e^{-\alpha}}+\frac{3}{2}k^{2}C_{2},\\
q<\frac{\alpha}{\alpha+1}<1,\\
  0<\frac{C_{2}k^{q+1}}{e^{\alpha-\alpha q}-1}<\frac{1}{3}.
  \end{cases}
\end{equation}
In what follows, we prove that
\begin{equation}\label{4.2}
  \|g_{k}(t,(t, \xi))-g_{k}(t,(t, \bar{\xi}))\|<\lambda\|\xi-\bar{\xi}\|^{q},\,0<\|\xi-\bar{\xi}\|<1, \, k\in \mathbb{N}.
\end{equation}
Clearly, \eqref{4.2} holds for $k=0$. Suppose that \eqref{4.2} holds for $k=m$.
It suffices to prove that  \eqref{4.2} holds for $k=m+1$. Note
\begin{equation*}
\begin{split}
&g_{m+1}(t,(t,\xi))-g_{m+1}(t,(t,\bar{\xi})) \\
=& \int_{-\infty}^t U(t,s)[f(s,Y(s,t,\xi)+g_{m}(s,(t,\xi)))
- f(s,Y(s,t,\bar{\xi})+g_{m}(s,(t,\bar{\xi})))]ds. \qquad \textbf{$(J_1)$}
\end{split}
\end{equation*}
We divide $J_{1}$ into two parts:
\[ J_{1}=\int_{-\infty}^{t-\tau} + \int_{t-\tau}^{t}\triangleq J_{11}+J_{12},\]
where $\tau=\frac{1}{\alpha+1} \ln \frac{1}{\|\xi-\bar{\xi}\|}$. Then by conditions ${\bf(A_{1})}$ and ${\bf(A_{2})}$, we deduce that
$$
\begin{aligned}
\|J_{11}\|  \leq& \int_{-\infty}^{t-\tau} k e^{-\alpha(t-s)} 2 \mu(s) d s \\
\leq& \frac{2 k C_{1}}{1-e^{-\alpha}}\|\xi-\bar{\xi}\|^{\frac{\alpha}{\alpha+1}}.
\end{aligned}
$$
Since Eq. \eqref{linear-eq} satisfies condition ${\bf(A_{1})}$, we have
\begin{equation}\label{4.3}
  \|Y(s,t,\xi)-Y(s,t,\bar{\xi})\|\leq  ke^{-\alpha(s-t)}\|\xi-\bar{\xi}\|.
\end{equation}
Furthermore, it follows from ${\bf(A_{1})}$, ${\bf(A_{3})}$ and \eqref{4.3} that
\[\begin{split}
 \|g_{m}(s,(t,\xi))-g_{m}(s,(t,\bar{\xi}))\| =&\|g_{m}(s,(s,Y(s,t,\xi)))-g_{m}(s,(s,Y(s,t,\bar{\xi})))\| \\
 \leq& \lambda\|Y(s,t,\xi)-Y(s,t,\bar{\xi})\|^{q} \\
 \leq& \lambda k^{q} \|\xi-\bar{\xi}\|^{q}\cdot e^{-\alpha q(s-t)},
\end{split}\]
and
\[\begin{split}
 \|J_{12}\| \leq & \int_{t-\tau}^{t}  k  e^{-\alpha(t-s)} r(s)[\|Y(s,t,\xi)-Y(s,t,\bar{\xi})\|+\|g_{m}(s,(t,\xi))-g_{m}(s,(t,\bar{\xi}))\| ]ds\\
 \leq & \int_{t-\tau}^{t}  k e^{-\alpha(t-s)} r(s)[ke^{-\alpha(s-t)}\|\xi-\bar{\xi}\|
   +\lambda k^{q} e^{-\alpha q(s-t)} \|\xi-\bar{\xi}\|^{q}]ds \\
 \leq & \int_{t-\tau}^{t} k^{2}r(s)\|\xi-\bar{\xi}\|ds+\int_{t-\tau}^{t}\lambda k^{q+1}r(s)e^{(\alpha-\alpha q)(s-t)} \|\xi-\bar{\xi}\|^{q}ds\\
 \leq & \sum_{m \in[0,[\tau]]} k^{2}\|\xi-\bar{\xi}\|\int_{t-\tau+m}^{t-\tau+m+1} r(s)ds\\
  &+\sum_{m \in[0,[\tau]]}\lambda k^{q+1}\|\xi-\bar{\xi}\|^{q}\int_{t-\tau+m}^{t-\tau+m+1} r(s)e^{(\alpha-\alpha q)(s-t)}ds \\
 \leq & k^{2}C_{2}\tau \|\xi-\bar{\xi}\|
   +\sum_{m \in[0,[\tau]]}\lambda k^{q+1}\|\xi-\bar{\xi}\|^{q}C_{2} e^{(\alpha q-\alpha)\tau}e^{(\alpha-\alpha q)(m+1)}\\
 \leq & k^{2}C_{2}e^{\tau} \|\xi-\bar{\xi}\|+\lambda k^{q+1}C_{2}e^{(\alpha q-\alpha)(\tau-1)}\|\xi-\bar{\xi}\|^{q}
 \cdot \frac{1-e^{(\alpha-\alpha q)[\tau]}}{1-e^{\alpha-\alpha q}} \\
 \leq & k^{2}C_{2}\|\xi-\bar{\xi}\|^{\frac{\alpha}{\alpha+1}}
  +\frac{\lambda C_{2} k^{q+1}}{e^{\alpha-\alpha q}-1}\|\xi-\bar{\xi}\|^{q}.
\end{split}\]
Notice that $\alpha-\alpha q>0$ implies that $\exp\{\alpha-\alpha q\}>1$ and $q<\frac{\alpha}{\alpha+1}$. Then
\[  \|J_{12}\| \leq  \left(C_{2} k^{2}+\frac{\lambda C_{2} k^{q+1}}{e^{\alpha-\alpha q}-1}\right)\|\xi-\bar{\xi}\|^{q}. \]
From \eqref{4.1}, we have
\[
\begin{split}
&\|g_{m+1}(t,(t, \xi))-g_{m+1}(t,(t, \bar{\xi}))\|\leq  \|J_{11}\|+\|J_{12}\| \\
\leq& \left(\frac{2 k C_{1}}{1-e^{-\alpha}}+C_{2} k^{2}+\frac{\lambda C_{2} k^{q+1}}{e^{\alpha-\alpha q}-1}\right) \|\xi-\bar{\xi}\|^{q} \\
\leq& \lambda\|\xi-\bar{\xi}\|^{q}.
\end{split}
\]
Thus inequality \eqref{4.2} holds for any $k\in\mathbb{N}$. Setting $k\rightarrow +\infty$, we have
\[\|g(t,(t, \xi))-g(t,(t, \bar{\xi}))\|\leq \lambda\|\xi-\bar{\xi}\|^{q}. \]
In view of $G=x=g,$
\[
\|G(t, y)-G(t, \bar{y})\|  \leq(1+\lambda)\|y-\bar{y}\|^{\beta}=p_{2}\|y-\bar{y}\|^{q}.
\]
\end{proof}

\noindent{\bf Step 4.}
We prove that the inverse of the equivalent function $G$ could not be Lipschitzian.
In fact, 
if we directly prove that $G$ is Lipschitzian by above similar arguments, then the following contradiction will appear:
\[ \begin{split}
J_{1}  &\leq  \int_{-\infty}^{t}  k  e^{-\alpha(t-s)} r(s)[\|Y(s,t,\xi)-Y(s,t,\bar{\xi})\|+\|g_{m}(s,(t,\xi))-g_{m}(s,(t,\bar{\xi}))\| ]ds\\
      &\leq \int_{-\infty}^{t}  k e^{-\alpha(t-s)} r(s)[ke^{-\alpha(s-t)}\|\xi-\bar{\xi}\|
   +\lambda k e^{-\alpha (s-t)} \|\xi-\bar{\xi}\|]ds \\
      &\leq  \int_{-\infty}^{t} k^{2}r(s)\|\xi-\bar{\xi}\|ds+\int_{-\infty}^{t}\lambda k^{2}r(s) \|\xi-\bar{\xi}\|ds.
\end{split}\]
Obviously, the integral $\int_{-\infty}^{t}r(s)ds$ is divergent in the above right-hand estimation, thus, $J_{1}$ is not Lipschitzian.
Consequently, the equivalent function $G$ could not be Lipschitzian.
The proof of main theorem is complete.

\begin{proof}[Proof of Theorem \ref{Cor3}.]
{\bf Step 1. proof of the map $H\in C^{1}$:}\\
  Note that
\[ h(t,(t,\xi))=-\int_{-\infty}^t U(t,s)B(s)X(s,t,\xi)ds.\]
Thus we can deduce that
\[ \frac{\partial h(t,(t,\xi))}{\partial \xi}=-\int_{-\infty}^{t} U(t,s)B(s) \frac{\partial X(s,t,\xi)}{\partial \xi}ds,  \]
and this integral is convergent, i.e., it is well defined. In fact, we need the following auxiliary statement:
\[ \frac{\partial X(t,0,\xi)}{\partial \xi}=U(t,0)+\int_{0}^{t} U(t,\tau)B(\tau)\frac{\partial X(\tau,0,\xi)}{\partial \xi}d\tau. \]
Since $\|U(t,\tau)\|\leq ke^{-\alpha(t-\tau)}$ and $\|B(\tau)\|\leq \delta$, we have
\[ \left\|\frac{\partial X(t,0,\xi)}{\partial \xi}\right\|\leq ke^{-\alpha t}+\int_{0}^{t} ke^{-\alpha(t-\tau)}\cdot \delta \cdot
\left\|\frac{\partial X(\tau,0,\xi)}{\partial \xi}\right\|d\tau. \]
Using the Bellman inequality, we obtain
\[ \left\|\frac{\partial X(t,0,\xi)}{\partial \xi}\right\| \leq ke^{(k\delta-\alpha)t}.\]
We now continue with the proof of $H \in C^{1}$.  Since
\[\begin{split}
 \left\| \frac{\partial h(t,(t,\xi))}{\partial \xi} \right\| \leq & \int_{-\infty}^{t}ke^{-\alpha(t-\tau)}\cdot \delta \cdot
                \left\| \frac{\partial X(s,t,\xi)}{\partial \xi}\right\|ds \\
                \leq & \int_{-\infty}^{t} k\delta e^{-\alpha(t-\tau)}\cdot ke^{(k\delta-\alpha)(s-t)}ds\\
                \leq & k,
\end{split}\]
$ \frac{\partial h(t,(t,\xi))}{\partial \xi}$ is well defined and $h\in C^{1}$. Hence, it is clear that $H \in C^{1}$.\\
{\bf Step 2.  H\"{o}lder regularity of  $G$:}
The proof is similar to that Theorem \ref{Thm-2}.
\end{proof}

\section{Conflict of Interest}
\hskip\parindent
The authors declare that they have no conflict of interest.

\section{Data Availability Statement}
My manuscript has no associated data. It is pure mathematics.

\section*{Contributions}
 We declare that all the authors have same contributions to this paper.


\begin{thebibliography}{999}
\addtolength{\itemsep}{-0.55ex}


\bibitem{Hartman1}
P. Hartman, A lemma in the theory of structural stability of differential equations, {\em Proc. Amer. Math. Soc.}, 11 (1960) 610--620.


\bibitem{Grobman1}
D. Grobman, Homeomorphisms of systems of differential equations, {\em Dokl. Akad. Nauk SSSR}, 128 (1965) 880--881.

\bibitem{Pugh1}
C. Pugh, On a theorem of P. Hartman, {\em Amer. J. Math.}, 91  (1969) 363--367.

\bibitem{Farkas}   G. Farkas, A Hartman-Grobman result for retarded functional differential equations with an application to the numerics around hyperbolic equilibria, {\em Z. Angew. Math. Phys.}, 52 (2001), 421--432.

\bibitem{Lu1}
K. Lu, A Hartman-Grobman theorem for scalar reaction diffusion equations, {\em J. Differential Equations}, 93 (1991) 364--394.

\bibitem{Lu2}
P. Bates, K. Lu, A Hartman-Grobman theorem for the Cahn-Hilliard and phase-field equations, {\em J. Dyn. Differ. Equ.}, 6 (1994) 101--145.

\bibitem{Hein-Pruss1}
M. Hein, J. Pr\"{u}ss, The Hartman-Grobman theorem for semilinear hyperbolic evolution equations, {\em J. Differential Equations}, 261 (2016) 4709--4727.

\bibitem{Sternberg1}
S. Sternberg, Local $C^n$ transformations of the real line, {\em Duke Math. J.}, 24 (1957) 97--102.

\bibitem{Sternberg2}
S. Sternberg, Local contractions and a theorem of Poincar\'{e}, {\em Amer. J. Math.}, 79(1957) 809--824.


\bibitem{Sell1}
G. Sell, Smooth Linearization near a fixed point,  {\em Amer. J. Math.}, 107 (1985) 1035--1091.

\bibitem{Belitskii1}
G. Belitskii, Functional equations and the conjugacy of diffeomorphism of finite smoothness class, {\em Funct. Anal. Appl.}, 7 (1973) 268--277.

\bibitem{ElBialy1}
M. ElBialy, Local contractions of Banach spaces and spectral gap conditions, {\em J. Funct. Anal.}, 182 (2001) 108--150.


\bibitem{RS-JDE}
H. Rodrigues, J. Sol\`{a}-Morales, Smooth linearization for a saddle on Banach spaces, {\em J. Dyn. Differ. Equ.}, 16 (2004) 767--793.



\bibitem{ZWN-JFA} W. Zhang, W. Zhang, $C^1$ linearization for planar contractions, {\em J. Funct. Anal.}, 260 (2011) 2043-2063.

\bibitem{ZWN-MA} W. Zhang, W. Zhang, W. Jarczyk, Sharp regularity of linearization for $C^{1,1}$ hyperbolic diffeomorphisms, {\em  Math. Ann.},  358 (2014) 69--113.








\bibitem{ZWN-ETDS} W. Zhang, W. Zhang, $\alpha$-H\"older  linearization of hyperbolic diffeomorphisms with resonance, {\em Ergod. Theor. Dyn. Syst.}, 36 (2016) 310--334.

\bibitem{ZWN-JDE}
W. Zhang, W. Zhang, Sharpness for $C^1$ linearization of planar hyperbolic diffeomorphisms,  {\em J. Differential Equations},  257 (2014), 4470--4502.



\bibitem{Palmer1}
K. Palmer, A generalization of Hartman's linearization theorem, {\em J. Math. Anal. Appl.}, 41 (1973) 753--758.

\bibitem {Jiang1}
L. Jiang, Generalized exponential dichotomy and global linearization, {\em J. Math. Anal. Appl.}, 315 (2006) 474--490.

\bibitem{B-V}
L. Barreira, C. Valls, A Grobman-Hartman theorem for nonuniformly hyperbolic dynamics, {\em J. Differential Equations}, 228 (2006) 285--310.

\bibitem{B-V1}
L. Barreira, C. Valls, Conjugacies for linear and nonlinear perturbations of nonuniform behavior, {\em J. Funct. Anal.}, 253 (2007) 324-358.

\bibitem{B-V2}
L. Barreira, C. Valls, A Grobman-Hartman theorem for general nonuniform exponential dichotomies, {\em J. Funct. Anal.}, 257 (2009) 1976--1993.

\bibitem{B-V3}
L. Barreira, C. Valls, Conjugacies between linear and nonlinear non-uniform contractions,
{\em Ergod. Theor. Dyn. Syst.}, 28 (2008) 1--19.

\bibitem{B-V4} L. Barreira, C. Valls, Conjugacies between general contractions, {\em Linear Algebra Appl.}, 436 (2012) 3087--3098.


\bibitem{B-V5}
L. Barreira, C. Valls, H\"{o}lder conjugacies for random dynamical systems, {\em Phys. D}, 223 (2006) 256--269.

\bibitem{Xia-BSM}
Y. Xia, R. Wang, K. Kou, O'Regan, On the linearization theorem for nonautonomous differential equations, {\em Bull. Sci. Math.},  139 (2015) 829--846.

\bibitem{Huerta2}
I. Huerta, Linearization of a nonautonomous unbounded system with nonuniform contraction: A spectral approach,
 {\em Discrete Contin. Dyn. Syst.}, 40 (2020) 5571--5590.

\bibitem{Huerta1}
\'{A}. Casta\~{n}eda, I. Huerta, Nonuniform almost reducibility of nonautonomous linear differential equations,
  {\em J. Math. Anal. Appl.}, 485 (2020), 123822.

{
\bibitem{Reinfelds-IJPAM}
A. Reinfelds, D. \v{S}teinberga, Dynamical equivalence of quasilinear equations, {\em Int. J. Pure Appl. Math.}, 98 (2015) 355-364.

\bibitem{BDK-JDE}
L. Backes, D. Dragi\v{c}evi\'{c}, K. Palmer, Linearization and H\"{o}lder continuity for nonautonomous systems,
{\em J. Differential Equations}, 297 (2021) 536--574.

\bibitem{BD-arXiv2}
L. Backes, D. Dragi\v{c}evi\'{c}, Smooth linearization of nonautonomous coupled systems, preprint, arXiv: 2202.12367.


\bibitem{BD-arXiv1}
L. Backes, D. Dragi\v{c}evi\'{c}, Multiscale linearization of nonautonomous systems, preprint, arXiv: 2203.03694.
}

\bibitem{Reinfelds1}
A. Reinfelds, L. Sermone, Equivalence of nonlinear differential equations with impulse effect in Banach space, {\em Latv. Univ. Zint. Raksti.},
              577 (1992) 68--73.

\bibitem{Fenner-Pinto}
J. Fenner, M. Pinto, On a Hartman linearization theorem for a class of ODE with impulse effect, {\em Nonlinear Anal.}, 38 (1999) 307--325.

\bibitem{Xia1}
Y. Xia, X. Chen, V. Romanovski, On the linearization theorem of Fenner and Pinto, {\em J. Math. Anal. Appl.}, 400 (2013) 439--451.

\bibitem{Potzche1}
C. P\"{o}tzche, Topological decoupling, linearization and perturbation on inhomogeneous time scales, {\em J. Differential Equations}, 245 (2008) 1210--1242.



\bibitem{Zou-Xia}
C. Zou, Y. Xia, M. Pinto, H\"{o}lder regularity of topological equivalence functions of DEPCAGs with unbounded nonlinear terms, {\em Sci. Sin. Math.},
49 (2019) 1--26 (In Chinese).

\bibitem{Pinto-depcag}
M. Pinto, G. Robledo, A Grobman-Hartman Theorem for Differential Equations with Piecewise Constant Arguments of Mixed Type,
  {\em Z. Anal. Anwend.}, 37 (2018) 101--126.




\bibitem{Shi-Zhang1}
J. Shi, J. Zhang, The Principle of Classification for Differential Equations,  Science Press, Beijing, 2003 (in Chinese).



 \bibitem{Sell-spectral}
J. Sacker, R. Sell, A spectral theory for linear differential systems, {\em J. Differential Equations}, 27 (1978) 320--358.

\bibitem{Xia-spectral}
Y. Xia, Y. Bai,  D. O'Regan, A new method to prove the nonunifrom dichotomy spectrum theorem in $\mathbb{R}^{n}$, {\em Proc. Amer. Math. Soc.},
 147 (2019)  3905--3917.


 \bibitem{Xia-Zhang} J. Chu, F. Liao, S. Siegmund, Y. Xia, W. Zhang, Nonuniform dichotomy spectrum and reducibility for nonautonomous equations, {\em Bull. Sci. Math.}, 139 (2015) 538--557.




\end{thebibliography}
\end{document}